\documentclass[a4paper,draft]{amsproc}
\usepackage{amssymb}
\usepackage{amscd} 
\theoremstyle{plain}
\newtheorem*{theoA}{Theorem A}
\newtheorem*{theoB}{Theorem B}
\newtheorem*{theoC}{Theorem C}

 \newtheorem{theo}{Theorem}[section]
 
 \newtheorem{lem}{Lemma}[section]
 
\theoremstyle{definition}
 \newtheorem{exm}{Example}[section]
 \newtheorem{ques}{Question}[section]
 \newtheorem{defi}{Definition}[section]
 
\theoremstyle{remark}
 \newtheorem{rem}{Remark}[section]
 
 \newcommand{\ol}{\overline}
\newcommand{\be}{\begin{equation}}
\newcommand{\ee}{\end{equation}}
\newcommand{\beas}{\begin{eqnarray*}}
\newcommand{\eeas}{\end{eqnarray*}}
\newcommand{\bea}{\begin{eqnarray}}
\newcommand{\eea}{\end{eqnarray}}

 \numberwithin{equation}{section}

\renewcommand{\leq}{\leqslant}
\renewcommand{\geq}{\geqslant}
\renewcommand{\setminus}{\smallsetminus}
\title[ On periodicity of a meromorphic function ...]{\LARGE  On periodicity of a meromorphic function when sharing two sets IM }

\numberwithin {equation}{section}
\date{}
\author{Molla Basir Ahamed}
\address{ Department of Mathematics, Kalipada Ghosh Tarai Mahavidyalya, West Bengal, 734014, India.}
\email{basir\_math\_kgtm@yahoo.com, bsrhmd2014@gmail.com.}
\begin{document}
\vspace{18mm} \setcounter{page}{1} \thispagestyle{empty}

\begin{abstract}
	In this paper, we have investigated the sufficient conditions for periodicity of meromorphic functions and obtained two results directly improving
 the result of  \emph{Bhoosnurmath-Kabbur} \cite{Bho & Kab-2013}, \emph{Qi-Dou-Yang} \cite{Qi & Dou & Yan-ADE-2012}  and \emph{Zhang} \cite{Zha-JMMA-2010}.  Let $\mathcal{S}_{1}=\left\{z:\displaystyle\int_{0}^{z-a}(t-a)^n(t-b)^4dt+1=0\right\}$ and  $\mathcal{S}_{2}=\bigg\{a,b\bigg\}$, where $n\geq 4(n\geq 3)$ be an integer.\emph{ Let $f(z)$ be a  non-constant meromorphic (entire) function satisfying $\ol E_{f(z)}(\mathcal{S}_j)=\ol E_{f(z+c)}(\mathcal{S}_j), (j=1,\;2)$
 then $f(z)\equiv f(z+c)$.} Some examples have been exhibited to show that, it is not necessary that meromorphic function should be of finite order  and also to show that the sets considered in the paper simply can't be replace by arbitrary sets. At the last section, we have posed an open question for the further improvement of the results of this paper.  \end{abstract}
\maketitle

\section{Introduction, Definitions and Results}
  We assume that the reader is familiar with the elementary Nevanlinna theory, see, e.g., \cite{Goldberg,Hay-1964,Laine-1993,Yan & Yi-2003}. Meromorphic functions are always non-constant, unless otherwise specified.
  For such a function $f$ and $a\in\mathbb{\ol C}=:\mathbb{C}\cup\{\infty\}$, each $z$ with $f(z)=a$ will be called $a$-point of $f$.  We will use  here some standard definitions and basic notations from this theory. In particular by $N(r,a;f)$ ($\ol N(r,a;f)$) we denote the counting function (reduced counting function) of $a$-points of meromorphic functions $f$, $T(r,f)$ is the Nevanlinna characteristic function of $f$ and $S(r,f)$ is used to denote each functions which is of smaller order than $T(r,f)$ when $r\rightarrow \infty$.\par  We also denote $\mathbb{C^{*}}:=\mathbb{C}\setminus\{0\}$. As for the standard notation in the uniqueness theory of
  meromorphic functions, suppose that $f$ and $g$ are meromorphic. Denoting $E_f(a)$ $(\ol E_f(a) )$, the set of all $a$-points of $f$ counting multiplicities (ignoring multiplicities). We say that two meromorphic functions $f$, $g$ share the value $a$ $CM$ $(IM)$ if $E_{f}(a)=E_{g}(a)$ $(\ol E_{f}(a)=\ol E_{g}(a))$.\par
  The classical results in the uniqueness theory of meromorphic functions are the five-point, resp. four-point, theorems
  due to Nevanlinna \cite{Nevanlinna-1929}: If two meromorphic functions $f$, $g$ share five distinct values in the extended complex plane $IM$, then $f\equiv g$. The beauty of this
  result lies in the fact that there is no counterpart of this result in the real function theory. Similarly, if two meromorphic functions $f,\;g$ share four distinct values in the extended complex plane $CM$, then
  $f\equiv T\circ g$, where $T$ is a Möbius transformation.\par Clearly these results initiated the study of uniqueness of two meromorphic functions $f$ and $g$. The study becomes more interesting if the function $g$ is related with $f$.
  \begin{defi}
  	For a non-constant meromorphic function $f$ and any set $\mathcal{S}\subset\mathbb{\ol C}$, we define \beas E_{f}(\mathcal{S})=\displaystyle\bigcup_{a\in\mathcal{S}}\bigg\{(z,p)\in\mathbb{C}\times\mathbb{N}:f(z)=a,\;\text{with multiplicity}\; p\bigg\}, \eeas \beas\ol E_{f}(\mathcal{S})=\displaystyle\bigcup_{a\in\mathcal{S}}\bigg\{(z,1)\in\mathbb{C}\times\{1\}:f(z)=a\bigg\}.\eeas
  \end{defi}
  \par If $E_{f}(\mathcal{S})=E_{g}(\mathcal{S})$ ($\ol E_{f}(\mathcal{S})=\ol E_{g}(\mathcal{S})$) then we simply say $f$ and $g$ share $\mathcal{S}$ Counting Multiplicities(CM) (Ignoring Multiplicities(IM)).\par
  Evidently, if $\mathcal{S}$ contains one element only, then it coincides with the usual definition of $CM (IM)$ sharing of values.
  \begin{defi}
  	For a non-constant meromorphic function $g$ and $a\in\mathbb{C}$, we define $\ol N_{(2}\left(r,\displaystyle\frac{1}{g-a}\right)$ the reduced counting function of those $a$-points of $g$ of multiplicities $\geq 2$.
  \end{defi}
  \par  In $1976$, \emph{Gross} \cite{Gross} precipitated the research instigating the set sharing problem with a more general set up made tracks various direction of research for the uniqueness theory.\par In connection with the question posed by \emph{Gross} in\cite{Gross}, a sprinkling number of results have been obtained by many mathematicians \cite{Ban-PMJ-2007,Ban-TMJ-2010,Bho & Dya-2011,Fan & Lah-IJM-2003,Yi & Lin-KMJ-2006,Zha & Xu-AML-2008} concerning the uniqueness of meromorphic functions sharing two sets. But in most of the preceding results, in the direction, one set has always been kept fixed as the set of poles of a meromorphic function.\par Recently set sharing corresponding to a function and its shift or difference operator have been given priority  by the researchers than that of the introductory one.\par

In what follows, $c$ always means a non-zero constant. For a non-constant meromorphic function, we define its shift and difference operator respectively by $f(z+c)$ and $\Delta_c f=f(z+c)-f(z)$.\par

Now-a-days among the researchers \cite{Bho & Kab-2013,che-che-2012,che-che-li-2012,Qi & Dou & Yan-ADE-2012,Zha-JMMA-2010}, an increasing amount of interest has been found to find the possible relationship between a meromorphic function $f(z)$ and its shift $f(z+c)$ or its difference $\Delta_cf$.\par 
At the earlier stage, several authors were devoted to find uniqueness problems between two meromorphic functions $f$ and $g$ sharing two sets. But in this particular direction, the first inspection for uniqueness of a meromorphic function and its shift was due to \emph{Zhang} \cite{Zha-JMMA-2010}.
\par In 2010, \emph{Zhang} \cite{Zha-JMMA-2010} obtained the following results.
\begin{theoA}\cite{Zha-JMMA-2010}
	Let $m\geq 2$, $n\geq 2m+4$ with $n$ and $n-m$ having no common factors. Let $a$ and $b$ be two non-zero constant such that the equation $w^n+aw^{n-m}+b=0$ has no multiple roots. Let $\mathcal{S}_1=\{w:w^n+aw^{n-m}+b=0\}$ and $\mathcal{S}_2=\{\infty\}$. Suppose that $f(z)$ is a non-constant meromorphic function of finite order. Then $E_{f(z)}(\mathcal{S}_j)=E_{f(z+c)}(\mathcal{S}_j)$ $(j=1,\; 2)$  imply that $f(z)\equiv f(z+c)$.
\end{theoA}	
\begin{rem}
	For meromorphic function, note that $\#(\mathcal{S}_1)=9$ when the nature of sharing is $CM$.
\end{rem}
\begin{theoB}\cite{Zha-JMMA-2010}
	Let $n\geq 5$ be an integer and let $a$, $b$ be two non-zero constants such that the equation $w^n+aw^{n-1}+b=0$ has no multiple roots. Denote $\mathcal{S}_1=\{w:w^n+aw^{n-1}+b=0\}$. Suppose that $f$ is a non-constant entire function of finite order. Then $E_{f(z)}(\mathcal{S}_1)=E_{f(z+c)}(\mathcal{S}_1)$ implies $f(z)\equiv f(z+c)$.
\end{theoB}
\begin{rem}
	For entire function, note that $\#(\mathcal{S}_1)=5$, when the nature of sharing is $CM$.
\end{rem}\par Thus we see that \emph{Zhang} obtained the results for meromorphic function with the cardinality of main range set as $9$ and for entire function as $5$.\par Later, \emph{Qi-Dou-Yang} \cite{Qi & Dou & Yan-ADE-2012} studied the case for $m=1$ in \emph{Theorem A} and with the aid of some extra supposition  and got $\#(\mathcal{S}_1)=6$ when the nture of sharing is $CM$.\par Afterworlds, \emph{Bhoosnurmath-Kabbur} \cite{Bho & Kab-2013} improved \emph{Theorem A} by reducing the lower bound of the cardinality of range set in a little different way and obtained the following result.
\begin{theoC}\cite{Bho & Kab-2013}
Let $n\geq 8$ be an integer and $c(\neq 0,\;1)$ is a constant such that the equation $P(w)=\displaystyle\frac{(n-1)(n-2)}{2}w^n-n(n-2)w^{n-1}+\frac{n(n-1)}{2}z^{n-2}-c$. Let us suppose that $S_1=\{w:P(w)=0\}$  and $S_2=\{\infty\}$.  Suppose that $f(z)$ is a non-constant meromorphic function of finite order. Then $E_{f(z)}(\mathcal{S}_j)=E_{f(z+c)}(\mathcal{S}_j)$ $(j=1,\;2)$  imply that $f(z)\equiv f(z+c)$.
\end{theoC}
\begin{rem}
	For  meromorphic function, we see that $\#(\mathcal{S}_1)=8$ when the nature of sharing is $CM$.
\end{rem}\par The worth noticing fact is that, the lower bound of the cardinality of the main range set for the meromorphic function has always been fixed to $ 8 $ without the help of any extra supposition.
\par So for the improvement of all the above mentioned results it quite natural to investigate in this direction.  \emph{Theorems A, B, C} really motivates oneself for further study in this direction by solving the following question.
\begin{ques}
Is it possible to diminish further the lower bound of the cardinalities of the main range sets in \emph{Theorem A}, \emph{B} and \emph{C} ? \end{ques}\par

We also note that no attempts have so far been made by any researchers, till now to the best of our knowledge, to relax the nature of sharing the sets. So the following question is inevitable.
\begin{ques}
Can we relax the nature of sharing the sets from $CM$ to $IM$ in \emph{Theorem A}, \emph{B} and \emph{C} ?\end{ques}\par Now it would be interesting to know what happens if we replace the set of poles $\{\infty\}$ by new set in \emph{Theorems A, B, C}.\par  
In all the above mentioned results, the respective authors have considered meromorphic function with \emph{finite ordered} and got their results.
 So a natural investigation query is that:  Are \emph{Theorems A, B, C} not valid for infinite ordered meromorphic function ?  \par The following examples show that \emph{Theorems A, B, C} are true for infinite ordered meromorphic functions also.
\begin{exm}
	Let $f(z)=\displaystyle\exp\left(1+\sin\left(\frac{2\pi z}{c}\right)\right)$. Clearly $f(z)$ and $f(z+c)$ share the corresponding sets $\mathcal{S}_1$ in \emph{Theorems A, B, C} and also the set $\mathcal{S}_2$ and hence holds the conclusion also.
\end{exm}
\begin{exm}
	Let $f(z)=\displaystyle\exp\left(1+\exp\left(\displaystyle\frac{2\pi i z}{c}\right)\right)$.
\end{exm} One can construct such examples plenty in numbers. Therefore, one natural question arises as follows: 
\begin{ques}
	Can we get a corresponding results like \emph{Theorem A, B, C} by  omitting the term  \emph{finite ordered} ? 
\end{ques} 
Answering all the questions affirmatively is the main motivation of writing this paper. In this paper, we have significantly diminished the cardinality of the main range set by modifying the set of poles by a new one. We have successfully relaxed the nature of sharing also. \par Throughout the paper, for an integer $n\geq 4$, we will denote by  \beas \mathcal{P}(z)=\displaystyle\int_{0}^{z-a}(t-a)^n(t-b)^4dt+1,\;\;\text{where}\;\;a, b\in\mathbb{C}\;\;\text{with}\;\;a\neq b.\eeas \par Following are the two main result of this paper.
\begin{theo}\label{t1.1} Let $\mathcal{S}_{1}=\left\{z:\mathcal{P}(z)=0\right\}$ and $\mathcal{S}_{2}=\bigg\{a,b\bigg\}$, where $a\in\mathbb{C^{*}}$, $n\geq 4$ be an integer. Let $f(z)$ be a non-constant meromorphic function satisfying  $\ol E_{f(z)}(\mathcal{S}_j)=\ol E_{f(z+c)}(\mathcal{S}_j),$ $(j=1,\;2)$ 
	then $f(z)\equiv f(z+c)$.
\end{theo}
\begin{rem}
	For non-entire meromorphic function, one may observe that $\#(\mathcal{S}_1)=9$ when the nature of sharing is $IM$.
\end{rem}
\begin{theo}\label{t1.2}
	Let $\mathcal{S}_{1}=\left\{z:\mathcal{P}(z)=0\right\}$ and  $\mathcal{S}_{2}=\bigg\{a,b\bigg\}$, where  $a\in\mathbb{C^{*}}$, $n\geq 2$ be an integer. Let $f(z)$ be a  non-constant entire function satisfying  $\ol E_{f(z)}(\mathcal{S}_j)=\ol E_{f(z+c)}(\mathcal{S}_j),$ $(j=1,\;2)$ 
	then $f(z)\equiv f(z+c)$.
\end{theo}
\begin{rem}
	For entire function, we see that $\#(\mathcal{S}_1)=7$ when the nature of sharing is $IM$.
\end{rem}
The following examples satisfy \emph{Theorems \ref{t1.1}} and \emph{\ref{t1.2}} for ``entire" as well as `` meromorphic"  functions.
\begin{exm}
	Let us suppose that  $f(z)=\displaystyle\left(\displaystyle\frac{\tan\left(\displaystyle\frac{\pi z}{c}\right)+\alpha}{\tan\left(\displaystyle\frac{\pi z}{c}\right)-\beta}\right)+\displaystyle\frac{\cos\left(\displaystyle\frac{2\pi z}{c}\right)+\gamma}{\sin\left(\displaystyle\frac{2\pi z}{c}\right)-\delta}$, where $\alpha,\;\beta,\;\gamma,\;\delta,\; c\in\mathbb{C^{*}}$. It is clear that $\ol E_{f(z)}(\mathcal{S}_j)=\ol E_{f(z+c)}(\mathcal{S}_j),$ $(j=1,\;2)$ in \emph{Theorem \ref{t1.1}} 
	and note that $f(z)\equiv f(z+c)$.
\end{exm}
\begin{exm}
	Let $f(z)=\displaystyle\frac{\alpha+\beta\displaystyle\sin^2\left(\frac{\pi z}{c}\right)}{\gamma-\delta\displaystyle\cos^2\left(\frac{\pi z}{c}\right)},$ where $p$ be an even positive integer, $\alpha,\;\beta,\;\gamma,\;\delta,\; c\in\mathbb{C^{*}}$. It is clear that $\ol E_{f(z)}(\mathcal{S}_j)=\ol E_{f(z+c)}(\mathcal{S}_j),$ $(j=1,\;2)$ in \emph{Theorem \ref{t1.1}} 
	and note that $f(z)\equiv f(z+c)$.
\end{exm}
\begin{exm}
	Let $f(z)=ae^{pz}+b\displaystyle\cos^2\left(\frac{\pi z}{c}\right),$ where $p$ be an even positive integer, $a,\;,b\;,c\in\mathbb{C^{*}}$  with $e^c=-1$. It is clear that $\ol E_{f(z)}(\mathcal{S}_j)=\ol E_{f(z+c)}(\mathcal{S}_j),$ $(j=1,\;2)$ in \emph{Theorem \ref{t1.2}} 
	 and note that $f(z)\equiv f(z+c)$.
\end{exm}\par The next examples shows that the set considered in \emph{Theorem \ref{t1.1}}\; for ``entire" and \emph{Theorem \ref{t1.2}}\; for ``meromorphic" functions respectively can not be replaced by arbitrary sets.
 \begin{exm}
 		Let us suppose that $\mathcal{S}_1=\{\zeta:\zeta^9-1=0\}$ and $\mathcal{S}_2=\{0,\infty\}$. Let $f(z)=\displaystyle\frac{ae^z}{b-d\sin^2\left(\displaystyle\frac{\pi z}{c}\right)}$. It is clear that $\ol E_{f(z)}(\mathcal{S}_j)=\ol E_{f(z+c)}(\mathcal{S}_j),$ $(j=1,\;2)$ in \emph{Theorem \ref{t1.1}}\; with $e^c=\zeta$ and $a,\;b,\;c,\;d\in\mathbb{C^{*}}$ 
 	and note that $f(z)\not\equiv f(z+c)$.
 \end{exm}
\begin{exm}
	Let us suppose that $\mathcal{S}_1=\{\zeta:\zeta^7-1=0\}$ and $\mathcal{S}_2=\{0,1\}$. Let $f(z)=\displaystyle\exp\left(\cos\left(\frac{\pi z}{c}\right)\right)$ or $\displaystyle\exp\left(\sin\left(\frac{\pi z}{c}\right)\right)$. Then $f(z+c)=\displaystyle\exp\left(-\cos\left(\frac{\pi z}{c}\right)\right)$ or $\displaystyle\exp\left(-\sin\left(\frac{\pi z}{c}\right)\right)$ respectively. It is clear that $\ol E_{f(z)}(\mathcal{S}_j)=\ol E_{f(z+c)}(\mathcal{S}_j),$ $(j=1,\;2)$ in \emph{Theorem \ref{t1.2}}\; 
	and note that $f(z)\not\equiv f(z+c)$.
\end{exm}
\begin{exm}
	Let $\mathcal{S}_1=\bigg\{-1,\;1,\;-i,\;0,\;i,\;-\displaystyle\frac{1}{\sqrt{2}},\frac{1}{\sqrt{2}}\bigg\}$ and $\mathcal{S}_2=\{-2,\;2\}$. Let $f(z)=e^z$. It is clear that $\ol E_{f(z)}(\mathcal{S}_j)=\ol E_{f(z+c)}(\mathcal{S}_j),$ $(j=1,\;2)$ in \emph{Theorem \ref{t1.2}}\; with $e^c=-1$, $c\in\mathbb{C^{*}}$ 
	and note that $f(z)\not\equiv f(z+c)$.
\end{exm} 

%-----------------------------------------------------------------------%

\section{Auxiliary sefinitions and Some Lemmas}
It was \emph{Fujimoto} \cite{6}, who first discovered a special property of a polynomial, reasonably called as critical injection property though initially \emph{Fujimoto} \cite{6} called it as property (H).
\begin{defi}
	Let $\mathcal{P}(w)$ be a non-constant monic polynomial. We call $\mathcal{P}(w)$ a uniqueness polynomial if $\mathcal{P}(f)\equiv c\mathcal{P}(g)$ implies $f\equiv g$ for any non-constant meromorphic functions $f$ and $g$ and any non-zero constant $c$. We also call $\mathcal{P}(w)$ a uniqueness polynomial in a broad sense if $\mathcal{P}(f)\equiv \mathcal{P}(g)$ implies $f\equiv g$.
\end{defi}\par Next we recall here the property (H) and critically injective polynomial.\par Let $\mathcal{P}(w)$ be a monic polynomial without multiple zero whose derivative has mutually distinct $k$-zeros $e_1, e_2,\ldots, e_k$ with the multiplicities $q_1, q_2, \ldots, q_k$ respectively.\par Now, the property $\mathcal{P}(e_l)\neq \mathcal{P}(e_m)$ for $1\leq l<m \leq k$ is a known as property (H) and a polynomial $\mathcal{P}(w)$ satisfying this property is called critically injective polynomial.\par 
Given meromorphic functions $f(z)$ and $f(z+c)$ we associate $\mathcal{F}$, $\mathcal{G}$ by 
\be\label{e2.1} \mathcal{F}=\mathcal{P}(f),\;\; \mathcal{G}=\mathcal{P}(f(z+c)), \ee  to $\mathcal{F}$, $\mathcal{G}$ we associate $\mathcal{H}$ and $\Phi$ by the following formulas
\bea\label{e2.2} \mathcal{H}&=&\frac{\left(\displaystyle\frac{1}{\mathcal{F}}\right)^{\prime\prime}}{\left(\displaystyle\frac{1}{\mathcal{F}}\right)^{\prime}}-\frac{\left(\displaystyle\frac{1}{\mathcal{G}}\right)^{\prime\prime}}{\left(\displaystyle\frac{1}{\mathcal{G}}\right)^{\prime}} =\left(\frac{\;\;\mathcal{F}^{\prime\prime}}{\mathcal{F}^{\prime}}-\frac{2\mathcal{F}^{\prime}}{\mathcal{F}}\right)-\left(\frac{\;\;\mathcal{G}^{\prime\prime}}{\mathcal{G}^{\prime}}-\frac{2\mathcal{G}^{\prime}}{\mathcal{G}}\right),\eea  \bea\label{e2.3}\Phi=\frac{\mathcal{F}^{\prime}}{\mathcal{F}}-\frac{\mathcal{G}^{\prime}}{\mathcal{G}}.\eea\par Before proceeding to the actual proofs, we recall a few lemmas that take an important role in the reasoning.
\begin{lem}\label{l2.1}\cite{Mok-1971} Let $ g $ be a non-constant meromorphic function and let \beas \mathcal{R^{\#}}(g)=\displaystyle\frac{\displaystyle\sum_{i=1}^{n}a_ig^i}{\displaystyle\sum_{j=1}^{m}b_jg^j}, \eeas be an irreducible rational function in $ g $ with constant coefficients $\{a_i\} $, $ \{b_j\}$, where $ a_n\neq 0 $ and $ b_m\neq 0 $. Then \beas T(r,\mathcal{R^{\#}}(g))=\max\{n,m\}\; T(r,g)+S(r,g). \eeas
\end{lem}
\begin{lem}\cite{6}\label{l2.2}
	Let $\mathcal{P}(w)$ be a polynomial satisfying the property \emph{(H)}. Then,  $\mathcal{P}(w)$ is a uniqueness polynomial in a broad sense if and only if \bea\label{e2.4}\sum_{1\leq l<m\leq k}q_{_l}q_{_m}>\sum_{l=1}^{k}q_{_l}.  \eea 
\end{lem}\par It can be easily verified that for the case $k\geq 4$, the condition (\ref{e2.4}) is always satisfied.  Moreover,  (\ref{e2.4}) holds when $\max\{q_1, q_2, q_3\}\geq 2$ for the case $k=3$ and when $\min\{q_1, q_2\}\geq 2$ and $q_1+q_2\geq 5$ for the case $k=2$. 

%-------------------------------------------------------------------------%

\section{Proofs of the theorems}
In this section, we give the proofs of our main results.
\begin{proof}[Proof of Theorem \ref{t1.1}] let $f(z)$ and $f(z+c)$ be any two non-constant meromorphic functions.  It is clear that \beas \mathcal{F}^{\prime}=(f(z)-a)^n(f(z)-b)^4f^{\prime}(z)\;\;\text{and}\;\;\mathcal{G}^{\prime}=(f(z+c)-a)^n(f(z+c)-b)^4f^{\prime}(z+c).\eeas We now discuss the following two cases:\par 
\indent{\bf{\sc Case 1.}} There exists a $\lambda>1$, $I\subset\mathbb{R}^{+}$ with measure of $I$ as $+\infty$ such that \bea\label{e3.1} && 2\ol N\left(r,\frac{1}{f(z)-a}\right)+2\ol N\left(r,\frac{1}{f(z)-b}\right)\\ &\geq&\nonumber \lambda\bigg\{T(r,f(z))+T(r,f(z+c))\bigg\}+S(r,f(z))+S(r,f(z+c)), \;\;(r\rightarrow +\infty, r\in I). \eea \par Let $\Phi$ is defined as in (\ref{e2.3}). Our aim is to show that $\Phi=0$. Let if possible $\Phi\not\equiv 0$. Then since  $n\geq 4$, so from the construction of $\Phi$, we get \bea\label{e3.2} 4\ol N\left(r,\frac{1}{f(z)-a}\right)+4\ol N\left(r,\frac{1}{f(z)-b}\right)\leq N\left(r,\frac{1}{\Phi}\right).\eea\par The possible poles of $\phi$ occur at the following points:
	(i) poles of $f(z)$, (ii) poles of $f(z+c)$, (iii) all the zeros of $\mathcal{F}$ of multiplicities $\geq 2$ and
 (iv) all the zeros of $\mathcal{G}$ of multiplicities $\geq 2$.
\par So we have \bea\label{e3.3}  N(r,\Phi)\leq \ol N(r.f(z))+\ol N_{(2}\left(r,\frac{1}{\mathcal{F}}\right)+\ol N(r.f(z+c))+\ol N_{(2}\left(r,\frac{1}{\mathcal{G}}\right). \eea\par By using \emph{First Fundamental Theorem} and (\ref{e3.2}), (\ref{e3.3}), we get \bea\label{e3.4} &&4\ol N\left(r,\frac{1}{f(z)-a}\right)+4\ol N\left(r,\frac{1}{f(z)-b}\right)\\ &\leq&\nonumber N\left(r,\frac{1}{\Phi}\right)\\ &\leq&\nonumber N(r,\Phi)\\ &\leq& \nonumber \ol N(r.f(z))+\ol N_{(2}\left(r,\frac{1}{\mathcal{F}}\right)+\ol N(r.f(z+c))+\ol N_{(2}\left(r,\frac{1}{\mathcal{G}}\right)\\\nonumber &&+S(r,f(z))+S(r,f(z+c)). \eea Again since $\ol E_{f(z)}(\mathcal{S}_2)=\ol E_{f(z+c)}(\mathcal{S}_2)$, so we must have \bea\label{e3.5} &&\ol N\left(r,\frac{1}{f(z)-a}\right)+\ol N\left(r,\frac{1}{f(z)-b}\right)\\ &=&\nonumber\ol N\left(r,\frac{1}{f(z+c)-a}\right)+\ol N\left(r,\frac{1}{f(z+c)-b}\right). \eea\par Now adding $\ol N\left(r,\displaystyle\frac{1}{\mathcal{F}}\right)+\ol N\left(r,\displaystyle\frac{1}{\mathcal{G}}\right)$ on both sides of (\ref{e3.4}), we get \bea\label{e3.6} && 4\ol N\left(r,\frac{1}{f(z)-a}\right)+4\ol N\left(r,\frac{1}{f(z)-b}\right)+ \ol N\left(r,\displaystyle\frac{1}{\mathcal{F}}\right)+\ol N\left(r,\displaystyle\frac{1}{\mathcal{G}}\right)\\ &\leq&\nonumber\ol N(r.f(z))+ N\left(r,\frac{1}{\mathcal{F}}\right)+\ol N(r.f(z+c))+ N\left(r,\frac{1}{\mathcal{G}}\right)\\\nonumber &&+S(r,f(z))+S(r,f(z+c)).\eea \par Next using (\ref{e3.5}) in (\ref{e3.6}), we get \bea\label{e3.7} && \bigg\{\ol N\left(r,\frac{1}{f(z)-a}\right)+\ol N\left(r,\frac{1}{f(z)-b}\right)\bigg\}+ 2\bigg\{\ol N\left(r,\frac{1}{f(z)-a}\right)+\ol N\left(r,\frac{1}{f(z)-b}\right)\bigg\}\\ &&\nonumber+\bigg\{\ol N\left(r,\frac{1}{f(z+c)-a}\right)+\ol N\left(r,\frac{1}{f(z+c)-b}\right)\bigg\}+ \ol N\left(r,\displaystyle\frac{1}{\mathcal{F}}\right)+\ol N\left(r,\displaystyle\frac{1}{\mathcal{G}}\right)\\ &\leq&\nonumber \ol N(r.f(z))+ N\left(r,\frac{1}{\mathcal{F}}\right)+\ol N(r.f(z+c))+ N\left(r,\frac{1}{\mathcal{G}}\right)+S(r,f(z))\\ &&+\nonumber S(r,f(z+c)). \eea \par By applying \emph{Second Fundamental Theorem}, we get \bea\label{e3.8} && (n+5)\bigg\{T(r,f(z))+T(r,f(z+c))\bigg\}\\ &\leq& \nonumber\ol N\left(r,\frac{1}{\mathcal{F}}\right)+\ol N\left(r,\frac{1}{f(z)-a}\right)+\ol N\left(r,\frac{1}{f(z)-b}\right)+ \ol N\left(r,\frac{1}{\mathcal{G}}\right)\\ &&+\ol N\left(r,\frac{1}{f(z+c)-a}\right)+\nonumber\ol N\left(r,\frac{1}{f(z+c)-b}\right)+S(r,f(z))+S(r,f(z+c)). \eea\par Adding  \beas 2 \ol N\left(r,\frac{1}{f(z)-a}\right)+2\ol N\left(r,\frac{1}{f(z)-b}\right) \eeas both sides in (\ref{e3.8}) and using (\ref{e3.7}), we get \beas && (n+5)\bigg\{T(r,f(z))+T(r,f(z+c))\bigg\}+ 2 \ol N\left(r,\frac{1}{f(z)-a}\right)+2\ol N\left(r,\frac{1}{f(z)-b}\right)\\ &\leq&  N\left(r,\frac{1}{\mathcal{F}}\right)+ N\left(r,\frac{1}{\mathcal{G}}\right)+\ol N(r,f(z))+\ol N(r,f(z+c))+S(r,f(z))+S(r,f(z+c))\\ &\leq& (n+6)\bigg\{T(r,f(z))+T(r,f(z+c))\bigg\}. \eeas i.e., \beas  2 \ol N\left(r,\frac{1}{f(z)-a}\right)+2\ol N\left(r,\frac{1}{f(z)-b}\right)\leq \bigg\{T(r,f(z))+T(r,f(z+c))\bigg\}, \eeas which is not possible for $\lambda>1$ in view of (\ref{e3.1}).\par Thus, we get $\Phi\equiv 0$. i.e., $\mathcal{F}\equiv \mathcal{AG}$, for $\mathcal{A}\in\mathbb{C}\setminus\{0\}$. Using \emph{Lemma \ref{l2.1}}, we have \bea\label{e3.9} T(r,f(z))=T(r,f(z+c))+S(r,f(z)). \eea
\indent{\bf{\sc Subcase 1.1.}} Let $\mathcal{A}\neq 1$.

%\indent{\bf{\sc Subcase 1.1.1.}} Let $\mathcal{A}\neq\mathcal{P}(b)$, where  \beas \mathcal{P}(b)&=&\int_{0}^{b-a}(t-a)^n(t-b)^4dt+1\\ &=&\sum_{j=0}^{4}\frac{^4C_j}{n+5-j}\bigg[\left(b-2a\right)^{n+5-j}-\left(-a\right)^{n+5-j}\bigg]+1\\ &\neq& 1,\;\; \text{as}\;\; a\neq b.\eeas

\par So from the relation $\mathcal{F}\equiv \mathcal{AG}$, we get \bea\label{e3.10} \mathcal{F}-\mathcal{A}\equiv\mathcal{A}(\mathcal{G}-1).  \eea\par A simple calculation shows that the polynomial $\mathcal{P}(z)-\mathcal{A}$ has all simple distinct roots and let them be $\sigma_{j}$ $(j=1, 2, \ldots, n+5)$ and all $\sigma_{j}\neq a,\;b$. Also we note that the polynomial $\mathcal{P}(z)-1$ has roots as $a$ of multiplicity $n+1$ and rest are $\delta_{j}$ $(j=1, 2, 3, 4)$.\par Thus we see from (\ref{e3.10}) that \bea &&\label{e3.11} \sum_{j=1}^{n+5}\ol N\left(r,\frac{1}{f(z)-\sigma_{j}}\right)\\ &=&\nonumber\ol N\left(r,\frac{1}{f(z+c)-a}\right)+\sum_{j=1}^{4}\ol N\left(r,\frac{1}{f(z+c)-\delta_{j}}\right). \eea\par Now by using \emph{Second Fundamental Theorem} and (\ref{e3.9}), we have \beas  (n+3)T(r,f(z)) &\leq& \sum_{j=1}^{n+5}\ol N\left(r,\frac{1}{f(z)-\sigma_{j}}\right)+S(r,f(z))\\ &\leq& \ol N\left(r,\frac{1}{f(z+c)-a}\right)+\sum_{j=1}^{4}\ol N\left(r,\frac{1}{f(z+c)-\delta_{j}}\right)+S(r,f(z))\\&\leq& 5 T(r,f(z))+S(r,f(z)), \eeas which contradicts $n\geq 4$.\\

%\indent{\bf{\sc Subcase 1.1.2.}} Let $\mathcal{A}=\mathcal{P}(b)$.\\ Thus, we have $\mathcal{F}\equiv \mathcal{P}(b)\mathcal{G}$. i.e., \bea\label{e3.12} \mathcal{F}-1\equiv \mathcal{P}(b)\mathcal{G}-1. \eea\par One can easily verify that the polynomial $\mathcal{P}(b)\mathcal{P}(z)-1$ has $n+5$ distinct zeros and let them be $\beta_{j}$ $(j=1, 2, \ldots, n+5)$ distinct zeros where $\beta_{j}\neq a, b$. Therefore, from (\ref{e3.12}), we must have \bea && \sum_{j=1}^{n+5}\ol N\left(r,\frac{1}{f(z+c)-\beta_{j}}\right)\\ &=&\nonumber\ol N\left(r,\frac{1}{f(z)-a}\right)+\sum_{j=1}^{4}\ol N\left(r,\frac{1}{f(z)-\delta_{j}}\right).  \eea\par Now proceeding exactly same way as done in \emph{\sc Subcase 1.1.1}, we get a contradiction.\\
\indent{\bfseries{\sc Subcase 1.2.}} Let $\mathcal{A}=1$. i.e., we have $\mathcal{F}\equiv \mathcal{G}$. Thus we get $\mathcal{P}(f)\equiv \mathcal{P}(f(z+c))$. We see that the polynomial $\mathcal{P}(z)=\displaystyle\int_{0}^{z-a}(t-a)^n(t-b)^4dt+1$ satisfies the condition (H) and (\ref{e2.4}) since $\mathcal{P}^{\prime}(z)=(z-a)^n(z-b)^4$, $k=2$, $e_1=a$, $e_2=b$ and $q_1=n\geq 4$, $q_2=4$. We next see that $\min\{q_1, q_2\}=\min\{n, 4\}\geq 2$ and $q_1+q_2=n+4\geq 5$. Therefore by \emph{Lemma \ref{l2.2}}, we see that the polynomial $\mathcal{P}(z)$ is a uniqueness polynomial in a broad sense. Hence the relation $\mathcal{P}(f)\equiv \mathcal{P}(f(z+c))$ implies $f(z)\equiv f(z+c)$.\\
\indent{\bf{\sc Case 2.}}	There exists $I\subset\mathbb{R}^{+}$ such that measure of $I$ is $+\infty$ such that \bea\label{e3.14} &&2\ol N\left(r,\frac{1}{f(z)-a}\right)+2\ol N\left(r,\frac{1}{f(z)-b}\right)\\&\leq&\left(1+\frac{1}{1000}\right)\bigg\{T(r,(z)f)+T(r,f(z+c))\bigg\}+S(r,(z)f)+S(r,f(z+c).\nonumber  \eea\par
We claim that $\mathcal{H}\equiv 0$. Suppose that $\mathcal{H}\not\equiv 0$. Next in view of the definition $\mathcal{H}$, we see that \bea\label{e3.15} \ol N^{E}_{1)}\left(r,\frac{1}{\mathcal{F}}\right)=\ol N^{E}_{1)}\left(r,\frac{1}{\mathcal{G}}\right)\leq N\left(r,\frac{1}{H}\right).  \eea\par Next we see that the possible poles of $\mathcal{H}$ occur at the following points: (i) poles of $f(z)$, (ii) poles of $f(z+c)$, (iii) zeros of $f(z)$, (iv) $1$-points of $f(z)$, (v) all those zeros of $f^{\prime}(z)$ which are not the zeros of $f(z)(f(z)-1)$ and (vi)	all those zeros of $f^{\prime}(z+c)$ which are not the zeros of $f(z+c)(f(z+c)-1)$. Thus we get \bea\label{e3.16} && N(r,\mathcal{H})\leq \ol N(r,f(z))+\ol N\left(r,\frac{1}{f(z)-a}\right)+\ol N\left(r,\frac{1}{f(z)-b}\right)+\ol N(r,f(z+c))\\ &&+\ol N_{0}(r,0;f^{\prime}(z))+\ol N_{0}(r,0;f^{\prime}(z+c)),\nonumber\eea where $\ol N_0\left(r,\displaystyle\frac{1}{f^{\prime}(z)}\right)$ is the reduced counting function of all those zeros of $f^{\prime}(z)$ which are not the zeros of $(f(z)-a)(f(z)-b)$. Similarly  $\ol N_0\left(r,\displaystyle\frac{1}{f^{\prime}(z+c)}\right)$ is defined. \par Therefore using \emph{First Fundamental Theorem}, we get \bea\label{e3.17} \ol N^{E}_{1)}\left(r,\frac{1}{\mathcal{F}}\right)\nonumber &\leq& N\left(r,\frac{1}{\mathcal{H}}\right)\\ &\leq& N(r,\mathcal{H}) \\&\leq&\nonumber \ol N(r,f(z))+\ol N\left(r,\frac{1}{f(z)-a}\right)+\ol N\left(r,\frac{1}{f(z)-b}\right)+\ol N(r,f(z+c))\\ &&+\ol N_{0}(r,0;f^{\prime}(z))+\ol N_{0}(r,0;f^{\prime}(z+c)).\nonumber \eea\par We also note that \beas \ol N_{(2}\left(r,\frac{1}{\mathcal{F}}\right)\leq \ol N_{0}\left(r,\frac{1}{f^{\prime}(z)}\right),\;\; \ol N_{(2}\left(r,\frac{1}{\mathcal{G}}\right)\leq \ol N_{0}\left(r,\frac{1}{f^{\prime}(z+c)}\right).\eeas\par Next we define $\Psi(z)=\displaystyle\frac{f^{\prime}(z)}{[f(z)-a][(f(z)-b]}\frac{f^{\prime}(z+c)}{[f(z+c)-a][f(z+c)-b}$.\par From the definition of $\Psi$ and by using \emph{First Fundamental Theorem} and (\ref{e3.5}), we get \bea\label{e3.18} && N_{0}\left(r,\frac{1}{f^{\prime}(z)}\right)+N_{0}\left(r,\frac{1}{f^{\prime}(z+c)}\right)\\&\leq& \nonumber\ol N\left(r,\frac{1}{\Psi}\right)\\ \nonumber&\leq& \nonumber\ol N(r,\Psi)\nonumber\\ &\leq& \nonumber\ol N\left(r,\frac{1}{f(z)-a}\right)+\ol N\left(r,\frac{1}{f(z)-b}\right)+\ol N\left(r,\frac{1}{f(z+c)-a}\right)\\ \nonumber&&+\ol N\left(r,\frac{1}{f(z+c)-b}\right)+S(r,f(z))+S(r,f(z+c))\nonumber\\ &\leq& 2\ol N\left(r,\frac{1}{f(z)-a}\right)+2\ol N\left(r,\frac{1}{f(z)-b}\right)+S(r,f(z))+S(r,f(z+c)).\nonumber\eea\par Adding \beas \ol N_{(2}\left(r,\frac{1}{\mathcal{F}}\right)+\ol N_{(2}\left(r,\frac{1}{\mathcal{G}}\right)+\ol N\left(r,\frac{1}{f(z)-a}\right)+\ol N\left(r,\frac{1}{f(z)-b}\right)   \eeas both sides of (\ref{e3.17}), we get \bea\label{e3.19} &&\ol N^{E}_{1)}\left(r,\frac{1}{\mathcal{F}}\right)+ \ol N_{(2}\left(r,\frac{1}{\mathcal{F}}\right)+\ol N_{(2}\left(r,\frac{1}{\mathcal{G}}\right)+\ol N\left(r,\frac{1}{f(z)-a}\right)+\ol N\left(r,\frac{1}{f(z)-b}\right) \\&\leq& \nonumber\ol N(r,f(z))+2\ol N\left(r,\frac{1}{f(z)-a}\right)+2\ol N\left(r,\frac{1}{f(z)-b}\right)+\ol N(r,f(z+c))\\ &&+2\ol N_{0}\left(r,\frac{1}{f^{\prime}(z)}\right)+2\ol N_{0}\left(r,\frac{1}{f^{\prime}(z+c)}\right). \nonumber\eea i.e., \bea\label{e3.20} &&\ol N\left(r,\frac{1}{\mathcal{F}}\right)+ \ol N\left(r,\frac{1}{f(z)-a}\right)+\ol N\left(r,\frac{1}{f(z)-b}\right)\\ &\leq& \ol N(r,f(z))+6\ol N\left(r,\frac{1}{f(z)-a}\right)+6\ol N\left(r,\frac{1}{f(z)-b}\right)+\ol N(r,f(z+c))\nonumber\\ &&+S(r,f(z))+S(r,f(z+c)).\nonumber \eea Similarly, we get \bea\label{e3.21} &&\ol N\left(r,\frac{1}{\mathcal{G}}\right)+ \ol N\left(r,\frac{1}{f(z+c)-a}\right)+\ol N\left(r,\frac{1}{f(z+c)-b}\right)\\ &\leq& \ol N(r,f(z+c))+6\ol N\left(r,\frac{1}{f(z+c)-a}\right)+6\ol N\left(r,\frac{1}{f(z+c)-b}\right)+\ol N(r,f(z))\nonumber\\ &&+S(r,f(z))+S(r,f(z+c)).\nonumber \eea\par Now by applying \emph{Second Fundamental Theorem} and (\ref{e3.14}), (\ref{e3.20}) and (\ref{e3.21}), we get \beas && (n+5)\bigg\{T(r,f(z))+T(r,f(z+c))\bigg\}\\ &\leq& \ol N\left(r,\frac{1}{\mathcal{F}}\right)+\ol N(r,f(z))+\ol N\left(r,\frac{1}{f(z)-a}\right)+\ol N\left(r,\frac{1}{\mathcal{G}}\right)+\ol N(r,f(z+c))\\ &&+\ol N\left(r,\frac{1}{f(z+c)-a}\right)+S(r,f(z))+S(r,f(z+c))\\ &\leq& 2\ol N(r,f(z))+2\ol N(r,f(z+c))+6\ol N\left(r,\frac{1}{f(z)-a}\right)+6\ol N\left(r,\frac{1}{f(z+c)-a}\right)\\ &&+6\ol N\left(r,\frac{1}{f(z)-b}\right)+6\ol N\left(r,\frac{1}{f(z+c)-b}\right)+S(r,f(z))+S(r,f(z+c))\\ &\leq& \left(8+\frac{6}{1000}\right)\bigg\{T(r,f(z))+T(r,f(z+c))\bigg\}++S(r,f(z))+S(r,f(z+c)), \eeas which contradicts $n\geq 4$.\par Therefore, we have $\mathcal{H}\equiv 0$. Thus we get \bea\label{e3.22} \frac{1}{\mathcal{F}}\equiv\frac{\mathcal{A}}{\mathcal{G}}+\mathcal{B}, \eea where $\mathcal{A}(\neq 0), \mathcal{B}\in\mathbb{C}$. In view of \emph{Lemma \ref{l2.1}}, we see from (\ref{e3.22}) that \bea\label{e3.23} T(r,f(z))=T(r,f(z+c))+S(r,f(z)).  \eea
\indent{\bfseries{\sc Subcase 2.1.}} Let $\mathcal{B}\neq 0$. Thus we must have \beas \ol N(r,f(z))=\ol N(r,\mathcal{F})=\ol N\left(r,\frac{1}{\mathcal{G}+\displaystyle\frac{\mathcal{A}}{\mathcal{B}}}\right) &\geq& 3T(r,f(z+c))+S(r,f(z+c)),\eeas which is absurd in view of (\ref{e3.23}).\\
\indent{\bfseries{\sc Subcase 2.2.}} So we have $\mathcal{B}=0$. Therefore (\ref{e3.22}) reduces to  $\mathcal{G}=\mathcal{AF}$. Next proceeding exactly same way as done in \emph{\sc Subcase 1.1}, we get $f(z)\equiv f(z+c)$.
\end{proof}
\begin{proof}[Proof of Theorem \ref{t1.2}] Since $f(z)$ is a non-constant entire function, so we must have $\ol N(r,f(z))=0$ and hence $\ol N(r,f(z+c))=0$. Now keeping this in mind, the rest of the proof follows the proof of \emph{Theorem \ref{t1.1}}.

\end{proof}

%%%%%%%%%%%%%%%%%%%%%%%%%%%%%%%%%%%%%%%%%%%%%%%%%%%%%%%%%%%%%%%%%%%%%%%%%%%%%%%%%%%%%%%%%%%%%%%%%%%%%%%%%%%%%%%%%%%%%%%%%%%%%%%%%%%%%%%%%%%%%%%%%
\section{An Open question}

\begin{ques}
	What is the best possible cardinality of two set sharing problem for the uniqueness of a meromorphic function and its shift operator?
\end{ques}

\end{document}